\documentclass[11pt]{article}
\usepackage{amsmath}
\usepackage{amssymb}
\usepackage{amsthm}
\usepackage{mathabx}
\usepackage[usenames]{color}
\usepackage{amscd}
\usepackage{dsfont}
\usepackage{indentfirst}

\usepackage[colorlinks=true,linkcolor=blue,filecolor=red,
citecolor=webgreen]{hyperref}
\definecolor{webgreen}{rgb}{0,.5,0}

\def\N{{\mathds{N}}}
\def\Z{{\mathds{Z}}}
\def\P{{\mathds{P}}}
\def\1{{\bf 1}}
\def\nr{{\trianglelefteq}}

\def\lcm{\operatorname{lcm}}

\newtheorem{theorem}{Theorem}
\newtheorem{thm}[theorem]{Theorem}

\newtheorem{cor}[theorem]{Corollary}
\newtheorem{conj}[theorem]{Conjecture}

\newtheorem{prop}[theorem]{Proposition}

\begin{document}

\title{\bf The number of subgroups of the group $\Z_m\times \Z_n \times \Z_r \times \Z_s$}
\author{L\'aszl\'o T\'oth}
\date{}
\maketitle

\begin{abstract} We deduce direct formulas for the total number of subgroups and the number of subgroups of a given order of the
group $\Z_m \times \Z_n \times \Z_r \times \Z_s$, where $m,n,r,s\in \N$. The proofs are by some simple group theoretical and number theoretical
arguments based on Goursat's lemma for groups. Two conjectures are also formulated.
\end{abstract}

{\sl 2010 Mathematics Subject Classification}:  20K01, 20K27, 11A25, 11Y70

{\it Key Words and Phrases}: cyclic group, direct product, subgroup,
number of subgroups, Goursat's lemma, finite abelian group of rank
four

\section{Introduction}

Throughout the paper we use the following notation: $\N:=\{1,2,\ldots\}$, $\N_0:=\N \cup \{0\}$; $\P$ is the set of primes;
the prime power factorization of $n\in \N$ is $n=\prod_{p\in \P} p^{\nu_p(n)}$, where all but
a finite number of the exponents $\nu_p(n)$ are zero; $\Z_n$ denotes the additive group of residue classes modulo $n$ ($n\in
\N$); $\gcd(a,b)$ and $\lcm(a,b)$ denote the greatest common divisor and least common
multiple, respectively of $a,b\in \N$; $\varphi$ is Euler's arithmetic function.

For an arbitrary finite abelian group $G$ of order $n$ let $N(G)$
and $N(G;k)$ denote the total number of subgroups, respectively the
number of subgroups  of order $k$ of $G$ ($k\mid n$). Let
$G= \bigtimes_{p\in \P} G_p$
be the primary decomposition of $G$, where
$|G_p| = p^{\nu_p(n)}$ ($p\in \P$). Then
\begin{equation} \label{prod_N}
N(G)=\prod_{p\in \P} N(G_p),
\end{equation}
\begin{equation} \label{prod_N_k}
N(G;k)=\prod_{p\in \P} N(G_p;p^{\nu_p(k)}).
\end{equation}

Therefore, the problem of counting the subgroups of $G$ reduces to $p$-groups. It is known that the number of subgroups of a given
type in an abelian $p$-group is a polynomial in $p$ with positive integral coefficients, which can be represented in terms of the
gaussian coefficients. There are also various other formulas in the literature. We refer to the monograph
by Butler \cite{But1994} and the papers by Bhowmik \cite{Bho1996}, C\u{a}lug\u{a}reanu \cite{Cal2004}, Shokuev \cite{Sho1972},
Stehling \cite{Ste1992}.

It follows that in an abelian $p$-group the total number of subgroups and the number of subgroups of a given order are also polynomials in $p$ with
positive integral coefficients. However, it is a difficult task  to make explicit these polynomials, even in the case of $p$-groups of
small rank. The rank two polynomials are
\begin{equation} \label{pol_rank_two}
N(\Z_{p^a}\times \Z_{p^b})= \frac{(b-a+1)p^{a+2}-(b-a-1)p^{a+1}-(a+b+3)p+(a+b+1)}{(p-1)^2},
\end{equation}
and
\begin{equation} \label{pol_rank_two_order}
N(\Z_{p^a}\times \Z_{p^b};p^c)= \begin{cases} \frac{p^{c+1}-1}{p-1}, & c\le a\le
b,\\ \frac{p^{a+1}-1}{p-1}, & a\le c\le b,
\\ \frac{p^{a+b-c+1}-1}{p-1},
& a\le b\le c\le a+b,
\end{cases}
\end{equation}
representing the total number of subgroups, respectively the number
of subgroups of order $p^c$ of the rank two group $\Z_{p^a}\times
\Z_{p^b}$ with $1\le a\le b$, $0\le c \le a+b$. See
C\u{a}lug\u{a}reanu \cite{Cal2004}, Petrillo \cite[Prop.\
2]{Pet2011}, T\u{a}rn\u{a}uceanu \cite[Th.\ 3.3]{Tar2010}).

Now consider the group $\Z_m\times \Z_n$, where $m,n\in \N$. Note that $\Z_m\times \Z_n$ is isomorphic to $\Z_{\gcd(m,n)}\times \Z_{\lcm(m,n)}$,
and it has rank two, assuming that $\gcd(m,n)>1$. Recall that a finite abelian group of order $>1$ has rank $t$ if it is isomorphic to
$\Z_{n_1} \times \cdots \times \Z_{n_t}$, where $n_1,\ldots,n_t \in \N\setminus \{1\}$ and $n_j \mid n_{j+1}$ ($1\le j\le t-1$), which
is the invariant factor decomposition of the given group. Here the number $t$ is uniquely determined and represents the minimal number of
generators of the given group.

For every $m,n\in \N$ the number of all subgroups of $\Z_m\times \Z_n$ is
\begin{equation} \label{rank_two}
N(\Z_m \times \Z_n) = \sum_{\substack{i\mid m\\ j\mid n}} \gcd(i,j),
\end{equation}
and the number of its subgroups of order $k$ ($k \mid
mn$) is given by
\begin{equation} \label{rank_two_order_k}
N(\Z_m \times \Z_n;k)= \sum_{\substack{i\mid \gcd(m,k)\\ j\mid \gcd(n,k)\\ k\mid
ij}} \varphi \left(\frac{ij}{k}\right).
\end{equation}

Formulas \eqref{rank_two} and \eqref{rank_two_order_k}, established by the author \cite[Ths.\ 4.1, 4.3]{TotTatra2014},
generalize and put in more compact forms identities \eqref{pol_rank_two} and \eqref{pol_rank_two_order}, respectively. See also the paper
by Hampejs, Holighaus, the author, and Wiesmeyr \cite{HHTW2014} for a different approach.

The polynomial giving the total number of subgroups of the rank three group $\Z_{p^a}\times \Z_{p^b}\times \Z_{p^c}$ was obtained by
Oh \cite[Cor.\ 2.2]{Oh2013} and by T\u{a}rn\u{a}uceanu and the author \cite[Cor.\ 2.9]{TarTot}. Formulas similar to
\eqref{rank_two} and \eqref{rank_two_order_k}, concerning the numbers of all subgroups, respectively subgroups of order
$k$ of the group $\Z_m\times \Z_n\times \Z_r$ ($m,n,r,k\in \N, k\mid mnr$) were given by
Hampejs and the author \cite[Th.\ 2.2]{HamTot2013}.

The next step is to investigate rank four groups, which is the goal of the present paper. The special case of the group
$\Z_p \times \Z_p \times \Z_p \times \Z_{p^i}$ ($i\in \N$) was recently considered by Chew, Chin and Lim \cite{Che2015}.
Instead of $p$-groups we prefer to investigate the group $\Z_m \times \Z_n \times \Z_r \times \Z_s$ with $m,n,r,s\in \N$,
which has rank four, provided that $\gcd(m,n,r,s)>1$. We deduce direct formulas for
\begin{equation*}
N(m,n,r,s):=N(\Z_m \times \Z_n \times \Z_r \times \Z_s)
\end{equation*}
and
\begin{equation*}
N(m,n,r,s;k):=N(\Z_m \times \Z_n \times \Z_r \times \Z_s;k),
\end{equation*}
representing the total number of subgroups of the group $\Z_m \times \Z_n \times \Z_r \times \Z_s$ and the number of its subgroups of
order $k$ ($m,n,r,s,k\in \N, k\mid mnrs$). See Theorems \ref{Theorem_subgroups} and \ref{Theorem_subgroups_order_k}. As direct consequences,
in Corollaries \ref{Cor_subgroups} and \ref{Cor_subgroups_order_k} we derive formulas for
the number of subgroups of $p$-groups of rank four. We do not obtain the explicit form of the corresponding polynomials, but those
can be computed in special cases. In Section \ref{Sect_val} we present tables for such polynomials, computed using Corollaries
 \ref{Cor_subgroups} and \ref{Cor_subgroups_order_k}. We also formulate two conjectures.

For the proofs of our results we use some simple group theoretical and number theoretical arguments based on Goursat's lemma for groups
(Proposition \ref{prop_Goursat}). We remark that for every $m,n,r,s\in \N$,
\begin{equation} \label{multipl_N}
N(m,n,r,s) = \prod_{p\in \P} N(p^{\nu_p(m)},p^{\nu_p(n)}, p^{\nu_p(r)}, p^{\nu_p(s)}),
\end{equation}
\begin{equation*}
N(m,n,r,s;k) = \prod_{p\in \P} N(p^{\nu_p(m)},p^{\nu_p(n)}, p^{\nu_p(r)}, p^{\nu_p(s)}; p^{\nu_p(k)}),
\end{equation*}
which are direct consequences of the identities \eqref{prod_N} and \eqref{prod_N_k}. Here \eqref{multipl_N} shows that
$(m,n,r,s)\mapsto N(m,n,r,s)$ is a multiplicative arithmetic function of four variables. We refer to our survey paper
\cite{Tot2014} regarding this number theoretical concept.

In turns out that the function $n\mapsto N(n):=N(n,n,n,n)$ is multiplicative, that is
$N(n)= \prod_{p\in \P} N(p^{\nu_p(n)})$ for every $n\in \N$. A table of values for $N(n)$ is presented in Section \ref{Sect_val}.

\section{Preliminaries} \label{Prel}

Goursat's lemma for groups can be stated as follows. For its history, proof and
discussion see the references given in \cite{TotTatra2014}. See also the recent paper \cite{BSZ2015}.

\begin{prop}[Goursat's lemma for groups] \label{prop_Goursat}
Let $G$ and $H$ be arbitrary groups. Then there is a bijection
between the set $S$ of all subgroups of $G \times H$ and the set $T$
of all $5$-tuples $(A, B, C, D, \Psi)$, where $B\, \nr \, A \le G$,
$D\, \nr \, C\le H$ and $\Psi: A/B \to C/D$ is an isomorphism (here
$\le$ denotes subgroup and $\nr$ denotes normal subgroup). More
precisely, the subgroup corresponding to $(A, B, C, D, \Psi)$ is
\begin{equation*}
K= \{ (g,h)\in A\times C: \Psi(gB)=hD\}.
\end{equation*}
\end{prop}

We will need the following corollary.

\begin{cor} \label{cor_Goursat} Assume that $G$ and $H$ are finite groups and that the subgroup
$K$ of $G \times H$ corresponds to the $5$-tuple $(A_K, B_K, C_K, D_K, \Psi_K)$ under this bijection.
Then one has $|A_K|\cdot |D_K| =|K|=|B_K|\cdot |C_K|$.
\end{cor}

Proposition \ref{prop_Goursat} was used in our paper
\cite{TotTatra2014} to count the subgroups of the group $(\Z_m\times
\Z_n,+)$ ($m,n\in \N$). In fact, formulas \eqref{rank_two} and
\eqref{rank_two_order_k} are consequences of the next representation
result. For every $m,n\in \N$ let
\begin{equation} \label{def_J}
J_{m,n}:=\left\{(a,b,c,d,\ell)\in \N^5: a\mid m, b\mid a, c\mid n,
d\mid c, \frac{a}{b}=\frac{c}{d}, \ell \le \frac{a}{b}, \,
\gcd\left(\ell,\frac{a}{b} \right)=1\right\}.
\end{equation}

Note that here $\gcd(b,d)\cdot \lcm(a,c)=ad$ and $\gcd(b,d)\mid
\lcm(a,c)$.

For $(a,b,c,d,\ell)\in J_{m,n}$ define
\begin{equation*} \label{def_K}
K_{a,b,c,d,\ell}:= \left\{\left(i\frac{m}{a}, i\ell
\frac{n}{c}+j\frac{n}{d}\right): 0\le i\le a-1, 0\le j\le
d-1\right\}.
\end{equation*}

\begin{prop}[{\rm \cite[Th.\ 3.1]{TotTatra2014}}] \label{prop_repres} Let $m,n\in \N$.

i) The map $(a,b,c,d,\ell)\mapsto K_{a,b,c,d,\ell}$ is a bijection
between the set $J_{m,n}$ and the set of subgroups of $(\Z_m \times
\Z_n,+)$.

ii) The invariant factor decomposition of the subgroup
$K_{a,b,c,d,\ell}$ of order $ad$ is
\begin{equation*} \label{H_isom}
K_{a,b,c,d,\ell} \simeq \Z_{\gcd(b,d)} \times \Z_{\lcm(a,c)}.
\end{equation*}
\end{prop}

We need the following new result concerning the quotient group
$(\Z_m\times \Z_n)/K_{a,b,c,d,\ell}$. Its proof is included in Section \ref{Proofs}.

\begin{prop} \label{prop_quotient} Let $m,n\in \N$. For every subgroup $K_{a,b,c,d,\ell}$
defied above, we have the invariant factor decomposition
\begin{equation*} \label{isom}
(\Z_m\times \Z_n)/K_{a,b,c,d,\ell} \simeq \Z_{\gcd
\left(\frac{m}{a},\frac{n}{c}\right)} \times \Z_{\lcm
\left(\frac{m}{b},\frac{n}{d}\right)},
\end{equation*}
where $\gcd \left(\frac{m}{a},\frac{n}{c}\right) \mid \lcm
\left(\frac{m}{b},\frac{n}{d}\right)$.
\end{prop}

Let $F(m,n)$ denote the number of automorphisms of the group $\Z_m \times \Z_n$. For every $m,n\in \N$,
\begin{equation*}
F(m,n)=\prod_{p\in \P} F(p^{\nu_p(m)}, p^{\nu_p(n)}),
\end{equation*}
hence $(m,n)\mapsto F(m,n)$ is a multiplicative arithmetic
function of two variables. For prime powers $p^a, p^b$ one has, cf. \cite[Th.\
4.1]{Hil}, \cite[Cor.\ 3]{GG2008},
\begin{align*}
F(p^a,p^b)  =\begin{cases} p^{2a} \varphi(p^a) \varphi(p^b), \quad
0\le a<b, \\ p^a \varphi_2(p^a) \varphi(p^a), \quad 0\le a=b,
\end{cases}
\end{align*}
where $\varphi_2$ is the Jordan function of order $2$ given by $\varphi_2(n)=n^2\prod_{p\mid n} (1-\frac1{p^2})$. We deduce that
$F(1,p^b)=\varphi(p^b)$ for $b\in \N$ and
\begin{align*}
F(p^a,p^b)  =\begin{cases} p^{3a+b} \left(1-\frac1{p}\right)^2, \quad 1\le a<b, \\
p^{4a} \left(1-\frac1{p}\right)\left(1-\frac1{p^2}\right), \quad
1\le a=b.
\end{cases}
\end{align*}

\section{Main results}

We state our main results. Their proofs are included in Section \ref{Proofs}.

\begin{thm} \label{Theorem_subgroups} For every $m,n,r,s\in \N$ we have
\begin{align} \label{sum_4}
N(m,n,r,s)= \sum \varphi(x_3) \varphi(y_3) \varphi(z_3) \varphi(t_3)
F(u,v),
\end{align}
where the sum is over all
$(x_1,x_2,x_3,x_4,x_5,y_1,y_2,y_3,y_4,y_5,z_1,z_2,z_3,z_4,z_5,t_1,t_2,t_3,t_4,t_5)\in
\N^{20} $ satisfying the following $10$ conditions:

{\rm (1)}\ $x_1x_2x_3=m$,\quad

{\rm (2)}\ $x_3x_4x_5=n$,\quad

{\rm (3)}\ $y_1y_2y_3=\gcd(x_1,x_4)$, \quad

{\rm (4)}\ $y_3y_4y_5=x_3\lcm(x_1,x_4)$,

{\rm (5)}\ $z_1z_2z_3=r$, \quad

{\rm (6)}\ $z_3z_4z_5=s$, \quad

{\rm (7)}\ $t_1t_2t_3=\gcd(z_1,z_4)$, \quad

{\rm (8)}\ $t_3t_4t_5=z_3\lcm(z_1,z_4)$,

{\rm (9)}\ $\gcd(y_2,y_5)=\gcd(t_2,t_5)=:u$, \quad

{\rm (10)}\ $y_3\lcm(y_2,y_5)= t_3\lcm(t_2,t_5)=:v$.
\end{thm}

\begin{thm} \label{Theorem_subgroups_order_k} For every $k,m,n,r,s\in
\N$ such that $k\mid mnrs$, the number $N(m,n,r,s;k)$ is given by
the sum \eqref{sum_4} with the additional constraint

{\rm (11)} \ $x_1x_3x_4t_1t_3t_4=k$.
\end{thm}

\begin{cor} \label{Cor_subgroups} For every $a,b,c,d\in \N_0$ we have
\begin{align} \label{sum_4_primes}
N(p^a,p^b,p^c,p^d)= \sum \varphi(p^{x_3}) \varphi(p^{y_3}) \varphi(p^{z_3}) \varphi(p^{t_3})
F(p^u,p^v),
\end{align}
where the sum is over all
$(x_1,x_2,x_3,x_4,x_5,y_1,y_2,y_3,y_4,y_5,z_1,z_2,z_3,z_4,z_5,t_1,t_2,t_3,t_4,t_5)\in
\N_0^{20} $ satisfying the following $10$ conditions:

{\rm (i)}\ $x_1+x_2+x_3=a$,\quad

{\rm (ii)}\ $x_3+x_4+x_5=b$,\quad

{\rm (iii)}\ $y_1+y_2+y_3=\min(x_1,x_4)$, \quad

{\rm (iv)}\ $y_3+y_4+y_5=x_3+\max(x_1,x_4)$,

{\rm (v)}\ $z_1+z_2+z_3=c$, \quad

{\rm (vi)}\ $z_3+z_4+z_5=d$, \quad

{\rm (vii)}\ $t_1+t_2+t_3=\min(z_1,z_4)$, \quad

{\rm (viii)}\ $t_3+t_4+t_5=z_3+\max(z_1,z_4)$,

{\rm (ix)}\ $\min(y_2,y_5)=\min(t_2,t_5)=:u$, \quad

{\rm (x)}\ $y_3+\max(y_2,y_5)= t_3+ \max(t_2,t_5)=:v$.
\end{cor}

\begin{cor} \label{Cor_subgroups_order_k} For every $a,b,c,d,k\in \N_0$
with $k\le a+b+c+d$ the number $N(p^a,p^b,p^c,p^d;p^k)$ is given by
the sum \eqref{sum_4_primes} with the additional constraint

{\rm (xi)} \ $x_1+x_3+x_4+t_1+t_3+t_4=k$.
\end{cor}

\section{Proofs} \label{Proofs}

\begin{proof}[Proof of Proposition {\rm \ref{prop_quotient}}] Let $K:= K_{a,b,c,d,\ell}$. The group $\Z_m\times \Z_n$ is
generated by $(0,1)$ and $(1,0)$. Therefore, the quotient group
$\Z_m\times \Z_n/K$ is generated by $(0,1)+K$ and $(1,0)+K$.

First we show that the order of $(0,1)+K$ is $\frac{n}{d}$. Indeed, this
follows from the following properties:

$\bullet$ $t(0,1)\in K$ if and only if there is $(i,j)\in \{0,1,\ldots, a-1\}
\times \{0,1,\ldots,d-1\}$ such that $(0,t)=(i\frac{m}{a}, i\ell
\frac{n}{c} +j\frac{n}{d})$, where the last condition is equivalent, in turn, to $i=0$ and
$t=j\frac{n}{d}$;

$\bullet$ the least such $t\in \N$ is $\frac{n}{d}$.

Next we show that the order of $(1,0)+K$ is $\frac{m}{b}$. Indeed,
observe that

$\bullet$ $u(1,0)\in K$ if and only if there is $(i,j)\in \{0,1,\ldots, a-1\}
\times \{0,1,\ldots,d-1\}$ such that $(u,0)=(i\frac{m}{a}, i\ell
\frac{n}{c} +j\frac{n}{d})$;

$\bullet$ therefore, $u=i\frac{m}{a}$ and $i\ell
\frac{n}{c}+j\frac{n}{d} \equiv 0$ (mod $n$), that is
$\frac{n}{c}(i\ell +j\frac{c}{d}) \equiv 0$ (mod $n$), $i\ell
+j\frac{c}{d} \equiv 0$ (mod $c$), the latter is a linear congruence
in $j$ and it turns out that for a fixed $i$ it has solution in $j$
if $\gcd\left(\frac{c}{d},c\right)=\frac{c}{d} \mid i\ell$, that is
$\frac{c}{d }\mid i$, since $\gcd\left(\ell,\frac{c}{d}\right)=1$,
cf. \eqref{def_J};

$\bullet$ the least such $i\in \N$ is $\frac{c}{d}$;

$\bullet$ finally, $u=\frac{c}{d}\cdot \frac{m}{a}=\frac{a}{b}\cdot
\frac{m}{a}= \frac{m}{b}$, representing the order of $(1,0)+K$.

We deduce that the exponent of $(\Z_m \times \Z_n)/K$ is
$\lcm(\frac{m}{b}, \frac{n}{d})$.

Now, $\Z_m \times \Z_n \simeq \Z_{\gcd(m,n)} \times \Z_{\lcm(m,n)}$
is a group of rank $\le 2$. Therefore, the quotient group $(\Z_m
\times \Z_n)/K$ has also rank $\le 2$. That is, $(\Z_m \times
\Z_n)/K \simeq \Z_v \times \Z_w$ for unique $v$ and $w$, where
$v\mid w$ and $vw=\frac{mn}{ad}$. We obtain that the exponent of
$(\Z_m \times \Z_n)/K$ is $\lcm(\frac{m}{b}, \frac{n}{d})=w$. Hence
\begin{equation*}
v=\frac{mn}{adw}= \frac{mn}{ad\lcm (\frac{m}{b}, \frac{n}{d})}=
\frac{mn}{ad \lcm (\frac{mn}{bn}, \frac{mn}{dm})} = \frac{mn}{ad
\frac{mn}{\gcd(bn, dm)}}
\end{equation*}
\begin{equation*}
= \frac{\gcd(bn,dm)}{ad}= \gcd\left( \frac{n}{d}\cdot
\frac{b}{a},\frac{m}{a}\right) = \gcd\left( \frac{n}{d}\cdot
\frac{d}{c},\frac{m}{a}\right) = \gcd\left( \frac{m}{a},
\frac{n}{c}\right).
\end{equation*}

This completes the proof.
\end{proof}

\begin{proof}[Proof of Theorem {\rm \ref{Theorem_subgroups}}]
We apply Goursat's lemma (Proposition \ref{prop_Goursat}) for the groups $G=\Z_m \times \Z_n$ and
$H=\Z_r \times \Z_s$. Let
\begin{equation*}
A=K_{a,b,c,d,\ell} \le \Z_m \times \Z_n,
\end{equation*}
where $a\mid m, b\mid a, c\mid n, d\mid c,
\frac{a}{b}=\frac{c}{d}=e, 1\le \ell \le e, \gcd \left(\ell,e
\right)=1$. Here, by Proposition \ref{prop_repres},
\begin{equation*}
A \simeq \Z_{m_1} \times \Z_{n_1},
\end{equation*}
where $m_1=\gcd(b,d)$, $n_1=\lcm(a,c)$. Also, let
\begin{equation*}
B \simeq K_{a_1,b_1,c_1,d_1,\ell_1} \le A,
\end{equation*}
where $a_1\mid m_1, b_1\mid a_1, c_1\mid n_1, d_1\mid c_1,
\frac{a_1}{b_1}=\frac{c_1}{d_1}=e_1, 1 \le \ell_1 \le e_1, \gcd
\left(\ell_1,e_1 \right)=1$. We have, by using Proposition
\ref{prop_quotient},
\begin{equation} \label{A/B}
A/B \simeq \Z_u \times \Z_v,
\end{equation}
where $u=\gcd(m_1/a_1,n_1/c_1)$, $v=\lcm(m_1/b_1,n_1/d_1)$.

In a similar way, let
\begin{equation*}
C=K_{\alpha,\beta,\gamma,\delta,\lambda} \le \Z_r \times \Z_s,
\end{equation*}
where $\alpha \mid r, \beta \mid \alpha, \gamma \mid s, \delta \mid
\gamma, \frac{\alpha}{\beta}=\frac{\gamma}{\delta}=\epsilon, 1\le
\lambda \le \epsilon, \gcd \left(\lambda,\epsilon \right)=1$. Here,
\begin{equation*}
C \simeq \Z_{r_1} \times \Z_{s_1},
\end{equation*}
where $r_1=\gcd(\beta,\delta)$, $s_1=\lcm(\alpha,\gamma)$.
Furthermore, let
\begin{equation*}
D \simeq K_{\alpha_1,\beta_1,\gamma_1,\delta_1,\lambda_1} \le C,
\end{equation*}
with $\alpha_1\mid r_1, \beta_1\mid \alpha_1, \gamma_1\mid s_1,
\delta_1\mid \gamma_1,
\frac{\alpha_1}{\beta_1}=\frac{\gamma_1}{\delta_1}=\epsilon_1, 1 \le
\lambda_1 \le \epsilon_1, \gcd \left(\lambda_1,\epsilon_1
\right)=1$. Here,
\begin{equation} \label{C/D}
C/D \simeq \Z_{\eta} \times \Z_{\theta},
\end{equation}
where $\eta= \gcd(r_1/\alpha_1,s_1/\gamma_1)$,
$\theta=\lcm(r_1/\beta_1,s_1/\delta_1)$.

Assume that the quotient groups $A/B$ and $C/D$ are isomorphic. According to \eqref{A/B} and \eqref{C/D},
this holds if and only if $u=\eta$ and $v=\theta$. Then there are $F(u,v)$ isomorphisms $A/B \simeq C/D$.

It follows that
\begin{equation*}
N(m,n,r,s)=\sum F(u,v),
\end{equation*}
where the sum is over all $(a,b,c,d,\ell,a_1,b_1,c_1,d_1,\ell_1,
\alpha,\beta,\gamma,\delta,\lambda,\alpha_1,\beta_1,\gamma_1,\delta_1,\lambda_1)\in
\N^{20}$ satisfying the above properties.

More exactly, let $m=ax, a=by$, $n=cz, c=dt$, where $y=t=e$. Hence $m=bxe$,
$n=dze$. Also, $m_1=b_1x_1e_1$, $n_1=d_1z_1e_1$, where
$m_1=\gcd(b,d)$, $n_1=\lcm(a,c)= e\lcm(b,d)$. Similarly, let $r=\alpha X, \alpha=\beta Y$, $s=\gamma Z,
\gamma=\delta T$, where $Y=T=\epsilon$. Hence $r=\beta X\epsilon$, $s=\delta
Z\epsilon$. Also, $r_1=\beta_1X_1\epsilon_1$,
$s_1=\delta_1Z_1\epsilon_1$, where $r_1=\gcd(\beta,\delta)$,
$r_1=\lcm(\alpha,\gamma)= \epsilon \lcm(\beta,\delta)$.

We deduce that
\begin{equation*}
N(m,n,r,s)=\sum F(u,v),
\end{equation*}
where the sum is over all $(b,x,e,d,z,\ell, b_1,x_1,e_1,d_1,z_1,\ell_1, \beta,X,\epsilon,\delta,Z,\lambda, \beta_1,X_1,\epsilon_1,
\delta_1,Z_1,\lambda_1)\in \N^{24}$ satisfying the following conditions:

{\rm (1')}\ $bxe=m$,\quad

{\rm (2')}\ $dze=n$,\quad

{\rm (3')}\ $b_1x_1e_1=\gcd(b,d)$, \quad

{\rm (4')}\ $d_1z_1e_1= e \lcm(b,d)$,

{\rm (4'')}\ $1\le \ell \le e, \gcd(\ell,e)=1$,

{\rm (5')}\ $\beta X \epsilon=r$, \quad

{\rm (6')}\ $\delta Z\epsilon=s$, \quad

{\rm (7')}\ $\beta_1X_1\epsilon_1=\gcd(\beta,\delta)$, \quad

{\rm (8')}\ $\delta_1Z_1\epsilon_1=\epsilon\lcm(\beta,\delta)$,

{\rm (8'')}\ $1\le \ell_1 \le e_1, \gcd(\ell_1,e_1)=1$,

{\rm (9')}\ $\gcd(x_1,z_1)=\gcd(X_1,Z_1)=u$, \quad

{\rm (10')}\ $e_1\lcm(x_1,z_1)= \epsilon_1\lcm(X_1,Z_1)=v$,

{\rm (10'')}\ $1\le \lambda \le \epsilon, \gcd(\lambda,\epsilon)=1$,

{\rm (10''')}\ $1\le \lambda_1 \le \epsilon_1, \gcd(\lambda_1,\epsilon_1)=1$.

Note that, according to (4''), (8''), (10'') and (10'''), the variables $\ell$, $\ell_1$, $\lambda$ and $\lambda_1$ take $\varphi(e)$, $\varphi(e_1)$,
$\varphi(\epsilon)$, respectively $\varphi(\epsilon_1)$ distinct values, independently from the other variables.

Now formula \eqref{sum_4} follows by introducing the notation:

$x_1:=b$, $x_2:=x$, $x_3:=e$, $x_4:=d$, $x_5:=z$,
$y_1:=b_1$, $y_2:=x_1$, $y_3:=e_1$, $y_4:=d_1$, $y_5:=z_1$
$z_1:=\beta$, $z_2:=X$, $z_3:=\epsilon$, $z_4:=\delta$, $z_5:=Z$,
$t_1:=\beta_1$, $t_2:=X_1$, $t_3:=\epsilon_1$, $t_4:=\delta_1$, $t_5:=Z_1$.
\end{proof}

\begin{proof}[Proof of Theorem {\rm \ref{Theorem_subgroups_order_k}}]
According to Corollary \ref{cor_Goursat}, $k=|A||D|=|B||C|$. Here $|A|=ad=bde$, $|B|=a_1d_1=b_1d_1e_1$. Similarly,
$|C|=\alpha \delta = \beta \delta \epsilon$, $|D|=\alpha_1 \delta_1= \beta_1 \delta_1 \epsilon_1$.
Hence condition $k=|A||D|=|B||C|$ is equivalent to $bde\beta_1 \delta_1 \epsilon_1=k$ and to $x_1x_3x_4t_1t_3t_4=k$ with the changing of
notation of above. Note that relation $k=|B||C|$ is a consequence of the other constraints and needs not to be considered.
\end{proof}

\begin{proof}[Proof of Corollary {\rm \ref{Cor_subgroups}}]
Apply Theorem \ref{Theorem_subgroups} in the case $(m,n,r,s)=(p^a,p^b,p^c,p^d)$.
\end{proof}

\begin{proof}[Proof of Corollary {\rm \ref{Cor_subgroups_order_k}}]
Apply Theorem \ref{Theorem_subgroups_order_k} in the case $(m,n,r,s)=(p^a,p^b,p^c,p^d)$, $k:=p^k$.
\end{proof}

\section{Tables and conjectures} \label{Sect_val}

\subsection{Values of $N(p^a,p^b,p^c,p^d)$ and $N(p^a,p^b,p^c,p^d;p^k)$}

Consider the polynomials $N(p^a,p^b,p^c,p^d)$ and $N(p^a,p^b,p^c,p^d;p^k)$, where $p$ is prime, $1\le a\le b \le c \le d$,
$0\le k \le a+b+c+d=:n$. It is known that $N(p^a,p^b,p^c,p^d;p^k)= N(p^a,p^b,p^c,p^d;p^{n-k})$ for every
$0\le k\le n$. The following unimodality result is also known: $N(p^a,p^b,p^c,p^d;p^k)-N(p^a,p^b,p^c,p^d;p^{k-1})$ has
nonnegative coefficients for every $1\le k \le \lfloor n/2 \rfloor$. More generally, similar properties hold for every abelian $p$-group.
See Butler \cite{But1987}, Takegahara \cite{Tak1998}.

In what follows we give the polynomials representing the values of $N(p^a,p^b,p^c,p^d)$ and $N(p^a,p^b,p^c,p^d;p^k)$,
where $1\le a\le b \le c \le d\le 3$, also $a=b=c=d=4$, with $0\le k\le \lfloor n/2 \rfloor$, computed using Corollaries
\ref{Cor_subgroups} and \ref{Cor_subgroups_order_k}.

We also formulate the following conjectures, confirmed by these examples.

\begin{conj} For every $1\le a\le b \le c \le d$, the degree of the polynomial $N(p^a,p^b,p^c,p^d)$ is $2a+b+c$.
\end{conj}

\begin{conj} For every $1\le m$ the degree of the polynomial $N(p^m,p^m,p^m,p^m)$ is $4m$ and its leading coefficient is $1$.
\end{conj}

\[
\vbox{\offinterlineskip  \hrule \halign{ \strut \vrule  $\ # \
$ \hfill & \vrule  $\ #  $ \hfill  \vrule \cr
 N(p,p,p,p) & 5+ 3p+4p^2+3p^3+p^4 \cr \noalign{\hrule \hrule}
 N(p,p,p,p;1)& 1  \cr \noalign{\hrule}
 N(p,p,p,p;p) & 1+p+ p^2+p^3  \cr \noalign{\hrule}
 N(p,p,p,p;p^2) & 1+p+2p^2+p^3+p^4 \cr} \hrule}
\]

\[
\vbox{\offinterlineskip  \hrule \halign{ \strut \vrule  $\ # \
$ \hfill & \vrule  $\ #  $ \hfill  \vrule \cr
 N(p,p,p,p^2) & 6+ 4p+6p^2+6p^3+2p^4 \cr \noalign{\hrule \hrule}
 N(p,p,p,p^2;1)& 1  \cr \noalign{\hrule}
 N(p,p,p,p^2;p) & 1+p+ p^2+p^3  \cr \noalign{\hrule}
 N(p,p,p,p^2;p^2) & 1+p+2p^2+2p^3+p^4 \cr} \hrule}
\]

\[
\vbox{\offinterlineskip  \hrule \halign{ \strut \vrule  $\ # \
$ \hfill & \vrule  $\ #  $ \hfill  \vrule \cr
 N(p,p,p,p^3) & 7+ 5p+8p^2+9p^3+3p^4 \cr \noalign{\hrule \hrule}
 N(p,p,p,p^3;1)& 1  \cr \noalign{\hrule}
 N(p,p,p,p^3;p) & 1+p+ p^2+p^3  \cr \noalign{\hrule}
 N(p,p,p,p^3;p^2) & 1+p+2p^2+2p^3+p^4 \cr \noalign{\hrule}
 N(p,p,p,p^3;p^3) & 1+p+2p^2+3p^3+p^4  \cr} \hrule}
\]

\[
\vbox{\offinterlineskip  \hrule \halign{ \strut \vrule  $\ # \
$ \hfill & \vrule  $\ #  $ \hfill  \vrule \cr
 N(p,p,p^2,p^2) & 7+ 5p+8p^2+9p^3+6p^4+p^5 \cr \noalign{\hrule \hrule}
 N(p,p,p^2,p^2;1)& 1  \cr \noalign{\hrule}
 N(p,p,p^2,p^2;p) & 1+p+ p^2+p^3  \cr \noalign{\hrule}
 N(p,p,p^2,p^2;p^2) & 1+p+2p^2+2p^3+2p^4 \cr \noalign{\hrule}
 N(p,p,p^2,p^2;p^3) & 1+p+2p^2+3p^3+2p^4+p^5  \cr} \hrule}
\]

\[
\vbox{\offinterlineskip  \hrule \halign{ \strut \vrule  $\ # \
$ \hfill & \vrule  $\ #  $ \hfill  \vrule \cr
 N(p,p,p^2,p^3) & 8+ 6p+10p^2+12p^3+10p^4+2p^5 \cr \noalign{\hrule \hrule}
 N(p,p,p^2,p^3;1)& 1  \cr \noalign{\hrule}
 N(p,p,p^2,p^3;p) & 1+p+ p^2+p^3  \cr \noalign{\hrule}
 N(p,p,p^2,p^3;p^2) & 1+p+2p^2+2p^3+2p^4 \cr \noalign{\hrule}
 N(p,p,p^2,p^3;p^3) & 1+p+2p^2+3p^3+3p^4+p^5 \cr} \hrule}
\]

\[
\vbox{\offinterlineskip  \hrule \halign{ \strut \vrule  $\ # \
$ \hfill & \vrule  $\ #  $ \hfill  \vrule \cr
 N(p,p,p^3,p^3) & 9+ 7p+12p^2+15p^3+14p^4+6p^5+p^6 \cr \noalign{\hrule \hrule}
 N(p,p,p^3,p^3;1)& 1  \cr \noalign{\hrule}
 N(p,p,p^3,p^3;p) & 1+p+ p^2+p^3  \cr \noalign{\hrule}
 N(p,p,p^3,p^3;p^2) & 1+p+2p^2+2p^3+2p^4 \cr \noalign{\hrule}
 N(p,p,p^3,p^3;p^3) & 1+p+2p^2+3p^3+3p^4+2p^5 \cr \noalign{\hrule}
 N(p,p,p^3,p^3;p^4) & 1+p+2p^2+3p^3+4p^4+2p^5+p^6 \cr} \hrule}
\]

\[
\vbox{\offinterlineskip  \hrule \halign{ \strut \vrule  $\ # \
$ \hfill & \vrule  $\ #  $ \hfill  \vrule \cr
 N(p,p^2,p^2,p^2) & 8+ 6p+10p^2+12p^3+10p^4+6p^5+2p^6 \cr \noalign{\hrule \hrule}
 N(p,p^2,p^2,p^2;1)& 1  \cr \noalign{\hrule}
 N(p,p^2,p^2,p^2;p) & 1+p+ p^2+p^3  \cr \noalign{\hrule}
 N(p,p^2,p^2,p^2;p^2) & 1+p+2p^2+2p^3+2p^4+p^5 \cr \noalign{\hrule}
 N(p,p^2,p^2,p^2;p^3) & 1+p+2p^2+3p^3+3p^4+2p^5+p^6 \cr} \hrule}
\]

\[
\vbox{\offinterlineskip  \hrule \halign{ \strut \vrule  $\ # \
$ \hfill & \vrule  $\ #  $ \hfill  \vrule \cr
 N(p,p^2,p^2,p^3) & 9+ 7p+12p^2+15p^3+14p^4+11p^5+4p^6 \cr \noalign{\hrule \hrule}
 N(p,p^2,p^2,p^3;1)& 1  \cr \noalign{\hrule}
 N(p,p^2,p^2,p^3;p) & 1+p+ p^2+p^3  \cr \noalign{\hrule}
 N(p,p^2,p^2,p^3;p^2) & 1+p+2p^2+2p^3+2p^4+p^5 \cr \noalign{\hrule}
 N(p,p^2,p^2,p^3;p^3) & 1+p+2p^2+3p^3+3p^4+3p^5+p^6 \cr \noalign{\hrule}
 N(p,p^2,p^2,p^3;p^4) & 1+p+2p^2+3p^3+4p^4+3p^5+2p^6 \cr} \hrule}
\]

\[
\vbox{\offinterlineskip  \hrule \halign{ \strut \vrule  $\ # \
$ \hfill & \vrule  $\ #  $ \hfill  \vrule \cr
 N(p,p^2,p^3,p^3) & 10+ 8p+14p^2+18p^3+18p^4+16p^5+10p^6+2p^7 \cr \noalign{\hrule \hrule}
 N(p,p^2,p^3,p^3;1)& 1  \cr \noalign{\hrule}
 N(p,p^2,p^3,p^3;p) & 1+p+ p^2+p^3  \cr \noalign{\hrule}
 N(p,p^2,p^3,p^3;p^2) & 1+p+2p^2+2p^3+2p^4+p^5 \cr \noalign{\hrule}
 N(p,p^2,p^3,p^3;p^3) & 1+p+2p^2+3p^3+3p^4+3p^5+2p^6 \cr \noalign{\hrule}
 N(p,p^2,p^3,p^3;p^4) & 1+p+2p^2+3p^3+4p^4+4p^5+3p^6+p^7 \cr} \hrule}
\]

\[
\vbox{\offinterlineskip  \hrule \halign{ \strut \vrule  $\ # \
$ \hfill & \vrule  $\ #  $ \hfill  \vrule \cr
 N(p,p^3,p^3,p^3) & 11+ 9p+16p^2+21p^3+22p^4+21p^5+16p^6+9p^7+3p^8 \cr \noalign{\hrule \hrule}
 N(p,p^3,p^3,p^3;1)& 1  \cr \noalign{\hrule}
 N(p,p^3,p^3,p^3;p) & 1+p+ p^2+p^3  \cr \noalign{\hrule}
 N(p,p^3,p^3,p^3;p^2) & 1+p+2p^2+2p^3+2p^4+p^5 \cr \noalign{\hrule}
 N(p,p^3,p^3,p^3;p^3) & 1+p+2p^2+3p^3+3p^4+3p^5+2p^6+p^7 \cr \noalign{\hrule}
 N(p,p^3,p^3,p^3;p^4) & 1+p+2p^2+3p^3+4p^4+4p^5+4p^6+2p^7+p^8 \cr \noalign{\hrule}
 N(p,p^3,p^3,p^3;p^5) & 1+p+2p^2+3p^3+4p^4+5p^5+4p^6+3p^7+p^8 \cr} \hrule}
\]

\[
\vbox{\offinterlineskip  \hrule \halign{ \strut \vrule  $\ # \
$ \hfill & \vrule  $\ #  $ \hfill  \vrule \cr
 N(p^2,p^2,p^2,p^2) & 9+ 7p+12p^2+15p^3+14p^4+11p^5+9p^6+3p^7+p^8 \cr \noalign{\hrule \hrule}
 N(p^2,p^2,p^2,p^2;1)& 1  \cr \noalign{\hrule}
 N(p^2,p^2,p^2,p^2;p) & 1+p+ p^2+p^3  \cr \noalign{\hrule}
 N(p^2,p^2,p^2,p^2;p^2) & 1+p+2p^2+2p^3+2p^4+p^5+p^6 \cr \noalign{\hrule}
 N(p^2,p^2,p^2,p^2;p^3) & 1+p+2p^2+3p^3+3p^4+3p^5+2p^6+p^7 \cr \noalign{\hrule}
 N(p^2,p^2,p^2,p^2;p^4) & 1+p+2p^2+3p^3+4p^4+3p^5+3p^6+p^7+p^8 \cr} \hrule}
\]

\[
\vbox{\offinterlineskip  \hrule \halign{ \strut \vrule  $\ # \
$ \hfill & \vrule  $\ #  $ \hfill  \vrule \cr
 N(p^2,p^2,p^2,p^3) & 10+ 8p+14p^2+18p^3+18p^4+16p^5+16p^6+6p^7+2p^8 \cr \noalign{\hrule \hrule}
 N(p^2,p^2,p^2,p^3;1)& 1  \cr \noalign{\hrule}
 N(p^2,p^2,p^2,p^3;p) & 1+p+ p^2+p^3  \cr \noalign{\hrule}
 N(p^2,p^2,p^2,p^3;p^2) & 1+p+2p^2+2p^3+2p^4+p^5+p^6 \cr \noalign{\hrule}
 N(p^2,p^2,p^2,p^3;p^3) & 1+p+2p^2+3p^3+3p^4+3p^5+3p^6+p^7 \cr \noalign{\hrule}
 N(p^2,p^2,p^2,p^3;p^4) & 1+p+2p^2+3p^3+4p^4+4p^5+4p^6+2p^7+p^8 \cr} \hrule}
\]

\[
\vbox{\offinterlineskip  \hrule \halign{ \strut \vrule  $\ # \
$ \hfill & \vrule  $\ #  $ \hfill  \vrule \cr
 N(p^2,p^2,p^3,p^3) & 11+ 9p+16p^2+21p^3+22p^4+21p^5+23p^6+14p^7+6p^8+p^9 \cr \noalign{\hrule \hrule}
 N(p^2,p^2,p^3,p^3;1)& 1  \cr \noalign{\hrule}
 N(p^2,p^2,p^3,p^3;p) & 1+p+ p^2+p^3  \cr \noalign{\hrule}
 N(p^2,p^2,p^3,p^3;p^2) & 1+p+2p^2+2p^3+2p^4+p^5+p^6 \cr \noalign{\hrule}
 N(p^2,p^2,p^3,p^3;p^3) & 1+p+2p^2+3p^3+3p^4+3p^5+3p^6+2p^7 \cr \noalign{\hrule}
 N(p^2,p^2,p^3,p^3;p^4) & 1+p+2p^2+3p^3+4p^4+4p^5+5p^6+3p^7+2p^8 \cr \noalign{\hrule}
 N(p^2,p^2,p^3,p^3;p^5) & 1+p+2p^2+3p^3+4p^4+5p^5+5p^6+4p^7+2p^8+p^9 \cr} \hrule}
\]

\[
\vbox{\offinterlineskip  \hrule \halign{ \strut \vrule  $\ # \ $
\hfill & \vrule  $\ #  $ \hfill  \vrule \cr
 N(p^2,p^3,p^3,p^3) & 12+ 10p+18p^2+24p^3+26p^4+26p^5+30p^6+22p^7+16p^8+6p^9+2p^{10} \cr \noalign{\hrule \hrule}
 N(p^2,p^3,p^3,p^3;1)& 1  \cr \noalign{\hrule}
 N(p^2,p^3,p^3,p^3;p) & 1+p+ p^2+p^3  \cr \noalign{\hrule}
 N(p^2,p^3,p^3,p^3;p^2) & 1+p+2p^2+2p^3+2p^4+p^5+p^6 \cr \noalign{\hrule}
 N(p^2,p^3,p^3,p^3;p^3) & 1+p+2p^2+3p^3+3p^4+3p^5+3p^6+2p^7+p^8 \cr \noalign{\hrule}
 N(p^2,p^3,p^3,p^3;p^4) & 1+p+2p^2+3p^3+4p^4+4p^5+5p^6+4p^7+3p^8+p^9 \cr \noalign{\hrule}
 N(p^2,p^3,p^3,p^3;p^5) & 1+p+2p^2+3p^3+4p^4+5p^5+6p^6+5p^7+4p^8+2p^9+p^{10}  \cr} \hrule}
\]

\[
\vbox{\offinterlineskip  \hrule \halign{ \strut \vrule  $\ # \ $
\hfill & \vrule  $\ #  $ \hfill  \vrule \cr
 N(p^3,p^3,p^3,p^3) & 13+
 11p+20p^2+27p^3+30p^4+31p^5+37p^6+30p^7+26p^8+18p^9 \cr
                    & +9p^{10}+3p^{11}+p^{12}
 \cr \noalign{\hrule \hrule}
 N(p^3,p^3,p^3,p^3;1)& 1  \cr \noalign{\hrule}
 N(p^3,p^3,p^3,p^3;p) & 1+p+ p^2+p^3  \cr \noalign{\hrule}
 N(p^3,p^3,p^3,p^3;p^2) & 1+p+2p^2+2p^3+2p^4+p^5+p^6 \cr \noalign{\hrule}
 N(p^3,p^3,p^3,p^3;p^3) & 1+p+2p^2+3p^3+3p^4+3p^5+3p^6+2p^7+p^8+p^9 \cr \noalign{\hrule}
 N(p^3,p^3,p^3,p^3;p^4) & 1+p+2p^2+3p^3+4p^4+4p^5+5p^6+4p^7+4p^8+2p^9+p^{10} \cr \noalign{\hrule}
 N(p^3,p^3,p^3,p^3;p^5) & 1+p+2p^2+3p^3+4p^4+5p^5+6p^6+6p^7+5p^8+4p^9+2p^{10}+p^{11} \cr \noalign{\hrule}
 N(p^3,p^3,p^3,p^3;p^6) & 1+p+2p^2+3p^3+4p^4+5p^5+7p^6+6p^7+6p^8+4p^9+3p^{10}+p^{11}+p^{12}\cr} \hrule}
\]

\[
\vbox{\offinterlineskip  \hrule \halign{ \strut \vrule  $\ # \ $
\hfill & \vrule  $\ #  $ \hfill  \vrule \cr
 N(p^4,p^4,p^4,p^4) & 17+ 15p+28p^2+39p^3+46p^4+51p^5+65p^6+62p^7+66p^8+66p^9 \cr
                    & +58p^{10}+46p^{11}+35p^{12}+18p^{13}+9p^{14}+3p^{15}+p^{16}
 \cr \noalign{\hrule \hrule}
 N(p^4,p^4,p^4,p^4;1)& 1  \cr \noalign{\hrule}
 N(p^4,p^4,p^4,p^4;p) & 1+p+ p^2+p^3  \cr \noalign{\hrule}
 N(p^4,p^4,p^4,p^4;p^2) & 1+p+2p^2+2p^3+2p^4+p^5+p^6 \cr \noalign{\hrule}
 N(p^4,p^4,p^4,p^4;p^3) & 1+p+2p^2+3p^3+3p^4+3p^5+3p^6+2p^7+p^8+p^9 \cr \noalign{\hrule}
 N(p^4,p^4,p^4,p^4;p^4) & 1+p+2p^2+3p^3+4p^4+4p^5+5p^6+4p^7+4p^8+3p^9+2p^{10}+p^{11}+p^{12} \cr \noalign{\hrule}
 N(p^4,p^4,p^4,p^4;p^5) & 1+p+2p^2+3p^3+4p^4+5p^5+6p^6+6p^7+6p^8+6p^9+5p^{10}+4p^{11}+2p^{12} \cr
                        & +p^{13} \cr \noalign{\hrule}
 N(p^4,p^4,p^4,p^4;p^6) & 1+p+2p^2+3p^3+4p^4+5p^5+7p^6+7p^7+8p^8+8p^9+8p^{10}+6p^{11}+5p^{12} \cr
                        & +2p^{13}+p^{14} \cr \noalign{\hrule}
 N(p^4,p^4,p^4,p^4;p^7) & 1+p+2p^2+3p^3+4p^4+5p^5+7p^6+8p^7+9p^8+10p^9+9p^{10}+8p^{11}+6p^{12} \cr
                        & +4p^{13}+2p^{14}+p^{15} \cr \noalign{\hrule}
 N(p^4,p^4,p^4,p^4;p^8) & 1+p+2p^2+3p^3+4p^4+5p^5+7p^6+8p^7+10p^8+10p^9+10p^{10}+8p^{11}+7p^{12} \cr
                        & +4p^{13}+3p^{14}+p^{15}+p^{16} \cr} \hrule}
\]

\subsection{Values of $N(n)$}

The values of $N(n)$ for $1\le n\le 30$ are given by the next table.

\[
\vbox{\offinterlineskip \hrule \halign{ \strut
\vrule \hfill $\ # \ $ \hfill
& \vrule \hfill $\ # \ $
& \vrule \vrule  \hfill $\ # \ $ \hfill
& \vrule \hfill $\ # \ $
& \vrule \vrule \hfill $\ # \ $  \hfill
& \vrule \hfill $\ # \ $  \vrule \cr
 n & N(n) & n & N(n) & n & N(n) \ \cr \noalign{\hrule}
 1 & 1         & 11 & 19\, 156      & 21  & 774\, 224   \cr \noalign{\hrule}
 2 & 67        & 12 & 420\, 396     & 22  & 1\, 283\, 452   \cr \noalign{\hrule}
 3 & 212       & 13 & 35\, 872      & 23  & 318\, 532  \cr \noalign{\hrule}
 4 & 1\, 983   & 14 & 244\, 684     & 24  & 9\, 187\, 868   \cr \noalign{\hrule}
 5 & 1\, 120   & 15 & 237\, 440     & 25 & 810\, 969  \cr \noalign{\hrule}
 6 & 14\, 204  & 16 & 821\, 335     & 26 & 2\, 403\, 424   \cr \noalign{\hrule}
 7 & 3\, 652   & 17 & 99\, 472      & 27 & 2\, 222\, 704 \cr \noalign{\hrule}
 8 & 43\, 339  & 18 & 1\, 610\, 211 & 28 & 7\, 241\, 916  \cr \noalign{\hrule}
 9 & 24\, 033  & 19 & 152\, 404     & 29 & 783\, 904   \cr \noalign{\hrule}
 10 & 75\, 040 & 20 & 2\, 220\, 960 & 30 & 15\, 908\, 480  \cr} \hrule}
\]

\section{Acknowledgment} The author thanks M\'aty\'as Koniorczyk for
his assistance during computations.

\medskip

\noindent L\'aszl\'o T\'oth \\
Department of Mathematics \\
University of P\'ecs \\
Ifj\'us\'ag \'utja 6, H-7624 P\'ecs, Hungary
\\ E-mail: {\tt ltoth@gamma.ttk.pte.hu}

\end{document}